\documentclass{amsart}

\usepackage{graphicx}
\usepackage{amsmath}
\usepackage{amssymb}

\newtheorem{thm}{Theorem}%[section]
\theoremstyle{definition}

\theoremstyle{remark}
\theoremstyle{plain}

\numberwithin{equation}{section}

\def\RR{{\mathbb R}}

\def\ZZ{{\mathbb Z}}

\def\hatZZ{\widehat{\mathbb Z}}

\def\veca{{\text{\boldmath$a$}}}

\def\vecb{{\text{\boldmath$b$}}}

\def\vecm{{\text{\boldmath$m$}}}

\def\vecp{{\text{\boldmath$p$}}}
\def\vecr{{\text{\boldmath$r$}}}

\def\vecx{{\text{\boldmath$x$}}}

\def\vecy{{\text{\boldmath$y$}}}

\def\vecnull{{\text{\boldmath$0$}}}

\def\scrB{{\mathcal B}}
\def\scrC{{\mathcal C}}

\def\scrF{{\mathcal F}}

\def\scrH{{\mathcal H}}

\def\scrR{{\mathcal R}}

\def\fC{{\mathfrak C}}

\def\e{\mathrm{e}}

\def\GL{\operatorname{GL}}

\def\SL{\operatorname{SL}}
\def\ASL{\operatorname{ASL}}

\def\supp{\operatorname{supp}}

\def\vol{\operatorname{vol}}

\def\GamG{\Gamma\backslash G}

\def\GamH{\Gamma_H\backslash H}

\def\SLSL{\SL(d,\ZZ)\backslash\SL(d,\RR)}
\def\SLZ{\SL(d,\ZZ)}
\def\SLR{\SL(d,\RR)}

\def\trans{\,^\mathrm{t}\!}

\title{Horospheres and Farey fractions}
\author{Jens Marklof}
\address{School of Mathematics, University of Bristol,
Bristol BS8 1TW, U.K.}
\email{j.marklof@bristol.ac.uk}
\date{16 December 2009/29 March 2010. To appear in Contemporary Mathematics}

\begin{document}

\begin{abstract}
We embed multidimensional Farey fractions in large horospheres and explain under which conditions they become uniformly distributed in the ambient homogeneous space. This question has recently been investigated in the case of $\SLZ$ to prove the asymptotic distribution of Frobenius numbers. The present paper extends these studies to general lattices in $\SLR$.
\end{abstract}

\maketitle
%\tableofcontents

\section{Introduction \label{secIntro}}

Let $G:=\SLR$ and $\Gamma$ a lattice in $G$.
The right action
\begin{equation}
	\GamG \to \GamG, \qquad \Gamma M \mapsto \Gamma M \Phi^t,\qquad \Phi^t=\begin{pmatrix} \e^{-t} 1_{d-1} & \trans\vecnull \\ \vecnull & \e^{(d-1)t} \end{pmatrix}
\end{equation}
defines a flow on the homogeneous space $\GamG$. The subgroups generated by 
\begin{equation}
	n_+(\vecx)=\begin{pmatrix} 1_{d-1} & \trans\vecnull \\ \vecx & 1 \end{pmatrix},\qquad
	n_-(\vecx)=\begin{pmatrix} 1_{d-1} & \trans\vecx \\ \vecnull & 1 \end{pmatrix}  .
\end{equation}  
parametrize the {\em stable and unstable horospheres} of the flow $\Phi^t$ as $t\to\infty$. The classical equidistribution of large horospheres can be stated as follows; see Section 5 of \cite{partI} for a proof of this particular version.

\begin{thm}\label{equiThm0}
Let $\Gamma$ be a lattice in $\SLR$, $\lambda$ be a Borel probability measure on $\RR^{d-1}$, absolutely continuous with respect to Lebesgue measure, and let $f:\RR^{d-1}\times\GamG\to\RR$ be bounded continuous. Then
\begin{equation}
	\lim_{t\to\infty} \int_{\RR^{d-1}} f\big(\vecx,n_-(\vecx)\Phi^{t}\big)\, d\lambda(\vecx)  = \int_{\RR^{d-1}\times\GamG} f(\vecx,M) \, d\lambda(\vecx) \, d\mu_\Gamma(M) .
\end{equation}
\end{thm}

Here
$\mu_{\Gamma}$ denotes the Haar measure on $G=\SLR$, normalized so that it represents the unique right $G$-invariant probability measure on the homogeneous space $\GamG$. In the special case $\Gamma=\SLZ$ we have by Siegel's volume formula
\begin{equation} \label{siegel}
d\mu_{\SLZ}(M) \frac{dt}t = \big(\zeta(2)\zeta(3)\cdots\zeta(d)\big)^{-1}\;\det (X)^{-d}
\prod_{i,j=1}^d dX_{ij},
\end{equation}
where $X=(X_{ij})=t^{1/d}M \in \GL^+(d,\RR)$ with $M\in G$, $t>0$; cf.~\cite{siegel}.

In \cite{frobenius} I studied the case where the absolutely continuous measure $\lambda$ is replaced by equally weighted point masses at the elements of the Farey sequence
\begin{equation}
	\scrF_Q=\bigg\{ \frac{\vecp}{q} \in[0,1)^{d-1} : (\vecp,q)\in\hatZZ^d, \; 0<q\leq Q \bigg\} ,
\end{equation}
where $\hatZZ^d$ denotes the set of primitive lattice points 
\begin{equation}
	\hatZZ^d=\{ (m_1,\ldots,m_d)\in\ZZ^d : \gcd(m_1,\ldots,m_d)=1 \}.
\end{equation}
It will be notationally convenient to also allow noninteger $Q\in\RR_{\geq 1}$. Note that $\scrF_Q=\scrF_{[Q]}$ where $[Q]$ is the integer part of $Q$.

The discussion in \cite{frobenius} is restricted to the case $\Gamma=\SLZ$, and the purpose of the present note is to describe the situation for a general lattice. 

Define the subgroups
\begin{equation}\label{Hdef}
\begin{split}
	H & = \bigg\{ M\in G :  (\vecnull,1)M =(\vecnull,1) \bigg\}\\
	& = \bigg\{ \begin{pmatrix} A  & \trans\vecb \\ \vecnull & 1 \end{pmatrix} : A\in \SL(d-1,\RR),\; \vecb\in\RR^{d-1} \bigg\},
\end{split}
\end{equation}
and
\begin{equation}
	\Gamma_H = \SLZ \cap H  = \bigg\{ \begin{pmatrix} \gamma & \trans\vecm \\ \vecnull & 1 \end{pmatrix} : \gamma\in\SL(d-1,\ZZ),\; \vecm\in\ZZ^{d-1} \bigg\} .
\end{equation}
Note that $H$ and $\Gamma_H$ are isomorphic to $\ASL(d-1,\RR)$ and $\ASL(d-1,\ZZ)$, respectively. We normalize the Haar measure $\mu_H$ of $H$ so that it becomes a probability measure on $\GamH$; explicitly:
\begin{equation} \label{siegel2}
d\mu_H(M) = d\mu_{\SL(d-1,\ZZ)}(A)\, d\vecb, \qquad M=\begin{pmatrix} A & \trans\vecb \\ \vecnull & 1 \end{pmatrix}.
\end{equation}

\begin{thm}\label{equiThm1}
Let $\Gamma$ be a lattice in $\SLR$, $\sigma\in\RR$, $f:[0,1]^{d-1}\times\GamG\to\RR$ be bounded continuous, and $Q=\e^{(d-1)(t-\sigma)}$. 
\begin{enumerate}
	\item[(A)] If $\Gamma$ {\bf is not} commensurable with $\SLZ$, then 
\begin{equation}
	\lim_{t\to\infty} \frac{1}{|\scrF_Q|} \sum_{\vecr\in\scrF_Q} f\big(\vecr,n_-(\vecr)\Phi^{t}\big)  = \int_{[0,1]^{d-1}\times\GamG} f(\vecx,M) \, d\vecx \, d\mu_\Gamma(M) .
\end{equation}
	\item[(B)] If $\Gamma$ {\bf is} commensurable with $\SLZ$, then 
\begin{multline}\label{equiThm1eq}
	\lim_{t\to\infty} \frac{1}{|\scrF_Q|} \sum_{\vecr\in\scrF_Q} f\big(\vecr,n_-(\vecr)\Phi^{t}\big)  \\
	= d(d-1)\e^{d(d-1)\sigma} \int_{\sigma}^\infty \int_{[0,1]^{d-1}\times\GamH} \overline f(\vecx, M \Phi^{-s}) \, d\vecx \, d\mu_H(M) \, \e^{-d(d-1)s} ds 
\end{multline}
with $\Delta:=\Gamma\cap\SLZ$ and
\begin{equation}
	\overline f(\vecx, M):=\frac{1}{|\SLZ:\Delta|} \sum_{\gamma\in\SLZ/\Delta}    f(\vecx,\gamma \trans M^{-1}).
\end{equation}
\end{enumerate}
\end{thm}

Note that in the above theorem we place 
\begin{equation}\label{asymQ}
	|\scrF_Q| \sim \frac{Q^d}{d\,\zeta(d)} \qquad (Q\to\infty) 
\end{equation}
points on a horosphere of volume $\e^{d(d-1) t}$. The scaling $Q=\e^{(d-1)(t-\sigma)}$ thus ensures that the average density of points on the horosphere remains constant as $t\to\infty$. If instead we had taken a scaling such that $Q \e^{-(d-1)t}\to\infty$, the Farey points would equidistribute also in Case (B) on all of $\GamG$ with respect to $d\mu$.

An interesting application of Case (A) is the following.

{\sc Corollary.}
{\em Let $\Gamma=\SLZ$, $\sigma\in\RR$, $f:[0,1]^{d-1}\times\GamG\to\RR$ be bounded continuous, and $Q=\e^{(d-1)(t-\sigma)}$. If the matrix $A\in\GL(d-1,\RR)$ has at least one irrational coefficient, then
\begin{equation}
	\lim_{t\to\infty} \frac{1}{|\scrF_Q|} \sum_{\vecr\in\scrF_Q} f\big(\vecr,n_-(\vecr A)\Phi^{t}\big)  = \int_{[0,1]^{d-1}\times\GamG} f(\vecx,M) \, d\vecx \, d\mu_\Gamma(M) .
\end{equation}
}

Note that
\begin{equation}
n_-(\vecr A) \Phi^t =	\begin{pmatrix} \trans A & \vecnull \\ \vecnull & 1 \end{pmatrix} n_-(\vecr) \Phi^t \begin{pmatrix} \trans A^{-1} & \vecnull \\ \vecnull & 1 \end{pmatrix} .
\end{equation}
The corollary therefore follows from Theorem \ref{equiThm1} by choosing the test function
\begin{equation}
	f_A(\vecx,M):= f\bigg(\vecx, \begin{pmatrix} \trans A & \vecnull \\ \vecnull & 1 \end{pmatrix} M \begin{pmatrix} \trans A^{-1} & \vecnull \\ \vecnull & 1 \end{pmatrix}\bigg) ,
\end{equation}
which is left invariant under the action of the lattice
\begin{equation}
	\Gamma_A = \begin{pmatrix} \trans A & \vecnull \\ \vecnull & 1 \end{pmatrix} \SLZ \begin{pmatrix} \trans A^{-1} & \vecnull \\ \vecnull & 1 \end{pmatrix}.
\end{equation}
Since $A$ is irrational, $\Gamma_A$ is not commensurable with $\SLZ$, and hence Case (A) applies.

The proof of Theorem \ref{equiThm1}, Case (B) utilizes Theorem \ref{equiThm0}, and thus follows (as we will see) the same argument as in \cite{frobenius} for $\Gamma=\SLZ$. Although the answer looks simpler in Case (A), the proof is more involved. The central step is the following equidistribution statement. 

\begin{thm}\label{equiThm2}
Let $\Gamma,\Gamma'$ be two incommensurable lattices in $G=\SLR$, $\lambda$ be a Borel probability measure on $\RR^{d-1}$, absolutely continuous with respect to Lebesgue measure, and let $f:\RR^{d-1}\times \GamG\times\Gamma'\backslash G\to\RR$ be bounded continuous. Then
\begin{multline}
	\lim_{t\to\infty} \int_{\RR^{d-1}} f\big(\vecx,n_-(\vecx)\Phi^{t},n_-(\vecx)\Phi^{t}\big)\, d\lambda(\vecx)  \\ = \int_{\RR^{d-1}\times\GamG\times\Gamma'\backslash G} f(\vecx,M,M') \, d\lambda(\vecx) \, d\mu_\Gamma(M)\,d\mu_{\Gamma'}(M') .
\end{multline}
\end{thm}

We will prove Theorem \ref{equiThm2} in Section \ref{secTwo}, and Theorem \ref{equiThm1} in Section \ref{secThree}.

It is well known that the Farey fractions correspond to the cusps of the space of lattices, $\SLSL$. A further interesting generalization (which we will not discuss here) is therefore to replace the Farey sequence by 
\begin{equation}
	\scrF_Q=\bigg\{ \frac{\vecp}{q} \in \scrR : (\vecp,q)\in (\vecnull,1)\Gamma', \; 0<q\leq Q \bigg\} ,
\end{equation}
where $\Gamma$ is a lattice in $\SLR$ with the property that $\Gamma_H=H\cap \Gamma'$ is a lattice in $H$. The set $\scrR\subset\RR^{d-1}$ is the pre-image of a fundamental domain of $\Gamma_H$ in $H$ under the map $\vecx\mapsto n_-(\vecx)$.

\section{Proof of Theorem \ref{equiThm2}}\label{secTwo}

Theorem \ref{equiThm2} is a consequence of Shah's Theorem 1.4 in \cite{Shah96}, which in turn follows from Ratner's theorem on the classification of measures that are invariant under unipotent flows \cite{Ratner91}.

\begin{thm}\label{thmShah}
Let $\tilde G$ be a connected Lie group and let $\tilde\Gamma$ be a lattice in $\tilde G$. Suppose $\tilde G$ contains a Lie subgroup $H$ isomorphic to $\SLR$ (we denote the corresponding embedding by $\varphi:\SLR\to \tilde G$), such that the set $\tilde\Gamma\backslash\tilde\Gamma H$ is dense in $\tilde\Gamma\backslash\tilde G$. Let $\lambda$ be a Borel probability measure on $\RR^{d-1}$ %
which is absolutely continuous with respect to Lebesgue measure, and let $f:\tilde\Gamma\backslash\tilde G\to\RR$ be bounded continuous.
Then
\begin{equation}
	\lim_{t\to\infty} \int_{\RR^{d-1}} f(\varphi( n_-(\vecx) \Phi^t)) d\lambda(\vecx)
	= \int_{\tilde\Gamma\backslash\tilde G} f \, d\tilde\mu,
\end{equation}
where $\tilde \mu$ is the unique $\tilde G$-right-invariant probability measure on $\tilde\Gamma\backslash\tilde G$.
\end{thm}

To use this result for the proof of Theorem \ref{equiThm2}, we take 
\begin{equation}
	\tilde G=G\times G, \qquad \tilde\Gamma=\Gamma\times\Gamma', \qquad \tilde\mu=\mu_\Gamma\times\mu_{\Gamma'},
\end{equation}
\begin{equation}
	H=\{ (M,M) : M\in G\}
\end{equation}
and $\varphi$ the diagonal embedding. It follows from Ratner's theory \cite{Ratner91} that the closure of $\tilde\Gamma\backslash\tilde\Gamma H$ equals $\tilde\Gamma\backslash\tilde\Gamma K$ for some closed connected subgroup $K$ with $H\leq K\leq \tilde G$. If $H=K$ then $\Gamma$ and $\Gamma'$ are commensurable, which contradicts our assumption. 

Let us therefore suppose $H\neq K$. Then $K$ contains a subgroup $\{ 1 \} \times L=K\cap (\{1\}\times G)$ where $\{1\}<L\leq G$. Since 
\begin{equation}
	(g,g)(1,h)(g,g)^{-1}=(1, g h g^{-1})
\end{equation}
for all $g\in G$, $L$ is in fact normal in $G$. Let $Z$ denote the (finite) center of $G$. Since $G/Z$ is simple, this implies $LZ/Z=G/Z$ and thus $LZ=G$. Therefore $\{1\}\times G\subset \tilde K:=K\tilde Z$, where $\tilde Z=\{1\}\times Z$. Together with $H\subset K$ this yields $\tilde K=G\times G=\tilde G$. But since $K$, $\tilde G$ are connected and $\tilde Z$ is finite, we have in fact $K=\tilde G$. This completes the proof of Theorem \ref{equiThm2} when the test function $f(\vecx,M,M')$ is independent of $\vecx$. The general case follows from the same argument as in the proof of Theorem 5.3 in \cite{partI}.

\section{Proof of Theorem \ref{equiThm1}}\label{secThree}

In the following set $\Gamma'=\SLZ$. The proof is virtually identical to that of Theorem 6 in \cite{frobenius}, except for the application of Theorem \ref{equiThm2} rather than Theorem \ref{equiThm0} in the incommensurable case. It will in fact be easier to prove the following generalization of Theorem \ref{equiThm1}. Theorem \ref{equiThm1} is then obtained from Theorem \ref{equiThm1b} below by choosing a test function of the form $f(\vecx,M)=f(\vecx,M,M')$.

\begin{thm}\label{equiThm1b}
Let $\Gamma$ be a lattice in $\SLR$, $\sigma\in\RR$, $f:[0,1]^{d-1}\times\GamG\times\Gamma'\backslash G\to\RR$ be bounded continuous, and $Q=\e^{(d-1)(t-\sigma)}$. 
\begin{enumerate}
	\item[(A)] If $\Gamma$ {\bf is not} commensurable with $\Gamma'$, then 
\begin{multline}
	\lim_{t\to\infty} \frac{1}{|\scrF_Q|} \sum_{\vecr\in\scrF_Q} f\big(\vecr,n_-(\vecr)\Phi^{t},n_-(\vecr)\Phi^{t}\big)  \\ = d(d-1)\e^{d(d-1)\sigma} \int_{\sigma}^\infty  \int_{[0,1]^{d-1}\times\GamG\times\GamH} \widetilde f(\vecx,M, M' \Phi^{-s}) \times \\ \times d\vecx \, d\mu_\Gamma(M) \, d\mu_H(M') \, \e^{-d(d-1)s} ds
\end{multline}
with
\begin{equation}
	\widetilde f(\vecx,M,M'):=f(\vecx,M,\trans {M'}^{-1}).
\end{equation}
	\item[(B)] If $\Gamma$ {\bf is} commensurable with $\Gamma'$, then 
\begin{multline}\label{equiThm1eqb}
	\lim_{t\to\infty} \frac{1}{|\scrF_Q|} \sum_{\vecr\in\scrF_Q} f\big(\vecr,n_-(\vecr)\Phi^{t},n_-(\vecr)\Phi^{t}\big)  \\
	= d(d-1)\e^{d(d-1)\sigma} \int_{\sigma}^\infty \int_{[0,1]^{d-1}\times\GamH} \overline f(\vecx, M \Phi^{-s}, M \Phi^{-s}) \times \\ \times d\vecx \, d\mu_H(M) \, \e^{-d(d-1)s} ds 
\end{multline}
with $\Delta:=\Gamma\cap\Gamma'$ and
\begin{equation}
	\overline f(\vecx, M,M'):=\frac{1}{|\Gamma':\Delta|} \sum_{\gamma\in\Gamma'/\Delta} f(\vecx,\gamma \trans M^{-1},\trans {M'}^{-1}) .
\end{equation}
\end{enumerate}
\end{thm}

\begin{proof}
{\bf Step 0: Uniform continuity.} By choosing the test function
\begin{equation}
	f(\vecx,M,M')=f_0(\vecx,M\Phi^{-\sigma},M'\Phi^{-\sigma})
\end{equation}
with $f_0:[0,1]^{d-1}\times\GamG\times\Gamma'\backslash G\to\RR$ bounded continuous, it is evident that we only need consider the case $\sigma=0$. We may also assume without loss of generality that $f$, and thus also $\widetilde f$ and $\overline f$, have compact support. That is, there is $\scrC\subset G$ compact such that $\supp f,\supp\widetilde f\subset[0,1]^{d-1}\times \Gamma\backslash\Gamma\scrC\times\Gamma'\backslash\Gamma'\scrC$. The generalization to bounded continuous functions follows from a standard approximation argument.

Since $f$ is continuous and has compact support, it is uniformly continuous. That is, given any $\delta>0$ there exists $\epsilon>0$ such that for all $(\vecx,M_1,M_1'),(\vecx',M_2,M_2')\in [0,1]^{d-1}\times G\times G$,
\begin{equation}\label{epsi}
	\|\vecx-\vecx'\| < \epsilon, \qquad d(M_1,M_1') < \epsilon , \qquad d(M_2,M_2') < \epsilon 
\end{equation}
implies
\begin{equation}\label{delt}
	\big| f(\vecx,M_1,M_2) - f(\vecx',M_1',M_2') \big| < \delta .
\end{equation}
Here $d$ denotes a left-invariant Riemannian metric on $G$. In the following, we choose $d$ in such a way that 
\begin{equation}\label{dchoice}
	d\big(n_\pm(\vecx),n_\pm(\vecx')\big)\leq \| \vecx-\vecx'\| ,
\end{equation}
where $\|\,\cdot\,\|$ the standard euclidean norm.

The plan is now to first establish \eqref{equiThm1eqb} for the set
\begin{equation}
	\scrF_{Q,\theta}=\bigg\{ \frac{\vecp}{q} \in[0,1)^{d-1} : (\vecp,q)\in\hatZZ^d, \; \theta Q<q\leq Q \bigg\} ,
\end{equation}
for any $\theta\in(0,1)$. The constant $\theta$ will remain fixed until the very last step of this proof.

{\bf Step 1: Thicken the Farey sequence.} The plan is to reduce the statement to Theorem \ref{equiThm0} (in the commensurable case) or Theorem \ref{equiThm2} (in the incommensurable case). To this end, we thicken the set $\scrF_{Q,\theta}$ as follows: For $\epsilon>0$ (we will in fact later use the $\epsilon$ from Step 0), let
\begin{equation}\label{FQeps0}
	\scrF_Q^\epsilon =  \bigcup_{\vecr\in\scrF_{Q,\theta}+\ZZ^{d-1}} \big\{ \vecx\in \RR^{d-1} : \big\| \vecx-\vecr \big\| < \epsilon\e^{-dt} \big\}  .
\end{equation}
Note that $\scrF_Q^\epsilon$ is symmetric with respect to $\vecx\mapsto-\vecx$.
A short calculation yields
\begin{equation}\label{FQeps}
	\scrF_Q^\epsilon = \bigcup_{\veca\in\hatZZ^d} \big\{ \vecx\in \RR^{d-1} :  \veca\,n_+(\vecx)\Phi^{-t}\in \fC_\epsilon \big\} ,
\end{equation}
where
\begin{equation}\label{Ceps}
	\fC_\epsilon= \big\{ (y_1,\ldots,y_d) \in\RR^d : \| (y_1,\ldots,y_{d-1}) \| < \epsilon y_d , \; \theta <y_d\leq 1 \big\} .
\end{equation}
Let
\begin{equation}\label{Ha}
	\scrH_\epsilon=\bigcup_{\veca\in\hatZZ^d} \scrH_\epsilon(\veca),\qquad 
	\scrH_\epsilon(\veca) = \big\{ M\in G:\veca M \in \fC_\epsilon \big\} .
\end{equation}
The bijection (cf.~\cite{siegel})
\begin{equation}\label{bij}
	\Gamma_H\backslash\Gamma' \to \hatZZ^d, \qquad \Gamma_H \gamma \mapsto (\vecnull,1) \gamma 
\end{equation}
allows us to rewrite
\begin{equation}\label{HaHa}
	\scrH_\epsilon=\bigcup_{\gamma\in\Gamma_H\backslash\Gamma'} \scrH_\epsilon((\vecnull,1)\gamma)=\bigcup_{\gamma\in\Gamma'/\Gamma_H} \gamma\scrH_\epsilon^1 ,
	\qquad\text{with } \scrH_\epsilon^1= \scrH_\epsilon((\vecnull,1)).
\end{equation}
Now
\begin{equation}\label{HaHa2}
\begin{split}
	\scrH_\epsilon^1 = & \big\{ M\in G: (\vecnull,1) M \in \fC_\epsilon \big\} \\
	= & H \big\{ M_\vecy : \vecy \in \fC_\epsilon \big\}
\end{split}
\end{equation}
with $H$ as in \eqref{Hdef}, and $M_\vecy\in G$ such that $(\vecnull,1) M_\vecy=\vecy$. Since $\vecy\in\fC_\epsilon$ implies $y_d>0$, we may choose
\begin{equation}\label{My}
	M_\vecy=\begin{pmatrix} y_d^{-1/(d-1)} 1_{d-1} & \trans\vecnull \\ \vecy' & y_d\end{pmatrix}, \qquad \vecy'=(y_1,\ldots,y_{d-1}).
\end{equation}

{\bf Step 2: Prove disjointness.} We will now prove the following claim: {\em Given a compact subset $\scrC\subset G$, there exists $\epsilon_0>0$ such that 
\begin{equation}\label{C1}
	\gamma\scrH_\epsilon^1\cap \scrH_\epsilon^1\cap\Gamma'\scrC=\emptyset
\end{equation}
for every $\epsilon\in(0,\epsilon_0]$, $\gamma\in \Gamma'-\Gamma_H$.} 

To prove this claim, note that \eqref{C1} is equivalent to 
\begin{equation}\label{C2}
	\scrH_\epsilon((\vecp,q))\cap\scrH_\epsilon^1\cap\Gamma'\scrC=\emptyset
\end{equation}
for every $(\vecp,q)\in\hatZZ^d$, $(\vecp,q)\neq(\vecnull,1)$. For 
\begin{equation}\label{MAb}
	M=\begin{pmatrix} A & \trans\vecb \\ \vecnull & 1 \end{pmatrix} M_\vecy, \qquad
	M_\vecy=\begin{pmatrix} y_d^{-1/(d-1)} 1_{d-1} & \trans\vecnull \\ \vecy' & y_d\end{pmatrix},
\end{equation}
we have
\begin{equation}
	(\vecp,q) M
	=(\vecp A y_d^{-1/(d-1)}+(\vecp\trans\vecb+q)\vecy', (\vecp\trans\vecb+q)y_d),
\end{equation}
and thus $M\in\scrH_\epsilon((\vecp,q))\cap\scrH_\epsilon^1$ if and only if
\begin{equation}\label{A1}
	\|\vecp A y_d^{-1/(d-1)}+(\vecp\trans\vecb+q)\vecy'\|<\epsilon (\vecp\trans\vecb+q)y_d,
\end{equation}
\begin{equation}\label{A12}
	\theta <(\vecp\trans\vecb+q)y_d \leq 1,
\end{equation}
and 
\begin{equation}\label{A2}
	\| \vecy' \| < \epsilon y_d , \qquad \theta <y_d\leq 1 .
\end{equation}
Relations \eqref{A12} and \eqref{A2} imply $\|(\vecp\trans\vecb+q)\vecy'\|<\epsilon(\vecp\trans\vecb+q)y_d\leq \epsilon$ and so, by \eqref{A1}, $\|\vecp A y_d^{-1/(d-1)}\|<2\epsilon(\vecp\trans\vecb+q)y_d\leq 2\epsilon$. That is, $\|\vecp A \|<2\epsilon y_d^{1/(d-1)}$ and hence 
\begin{equation}\label{A2e}
\|\vecp A \|< 2\epsilon.	
\end{equation}
Let us now suppose $M\in\Gamma'\scrC$ with $\scrC$ compact. The set
\begin{equation}
	\scrC'=\scrC \big\{ M_\vecy^{-1} : \vecy \in \overline \fC_\epsilon \big\} 
\end{equation}
is still compact, by the compactness of $\overline \fC_\epsilon$ (the closure of $\fC_\epsilon$) in $\RR^d\setminus\{\vecnull\}$. In view of \eqref{MAb} we obtain 
\begin{equation}
	\begin{pmatrix} A & \trans\vecb \\ \vecnull & 1 \end{pmatrix} \in \Gamma'\scrC',
\end{equation}
and so $A\in \SL(d-1,\ZZ) \scrC_0$ for some compact $\scrC_0\subset \SL(d-1,\RR)$.

Mahler's compactness criterion then shows that
\begin{equation}
	I:=\inf_{A\in\SL(d-1,\ZZ)\scrC_0}\inf_{\vecp\in\ZZ^{d-1}\setminus\{\vecnull\}}\|\vecp A \|>0 .
\end{equation}
Now choose $\epsilon_0$ such that $0<2\epsilon_0<I$. Then \eqref{A2e} implies $\vecp=\vecnull$ and therefore $q=1$. The claim is proved.

{\bf Step 3: Apply equidistribution of large horospheres.} Step 2 implies that, for $\scrC\subset G$ compact, there exists $\epsilon_0>0$ such that for every $\epsilon\in(0,\epsilon_0]$
\begin{equation}\label{HaHaHa}
	\scrH_\epsilon\cap \Gamma'\scrC =\bigcup_{\gamma\in\Gamma'/\Gamma_H} \big( \gamma\scrH_\epsilon^1 \cap \Gamma'\scrC \big) 
\end{equation}
is a disjoint union. Hence, if $\chi_\epsilon$ and $\chi_\epsilon^1$ are the characteristic functions of the sets $\scrH_\epsilon$ and $\scrH_\epsilon^1$, respectively, we have
\begin{equation}\label{chi}
	\chi_\epsilon(M) = \sum_{\gamma\in\Gamma_H\backslash\Gamma'} \chi_{\epsilon}^1(\gamma M) ,
\end{equation}
for all $M\in\Gamma\scrC$.
Evidently $\scrH_\epsilon^1$ and thus $\scrH_\epsilon$ have boundary of $\mu$-measure zero.
We furthermore set $\widetilde\chi_\epsilon(M):=\chi_\epsilon(\trans M^{-1})$, and note that  
\begin{equation}
	\chi_\epsilon\big(n_+(\vecx)\Phi^{-t}\big)=\chi_\epsilon\big(n_+(-\vecx)\Phi^{-t}\big)
\end{equation}
is the characteristic function of the set $\scrF_Q^\epsilon$; recall \eqref{FQeps} and the remark after \eqref{FQeps0}.
Therefore
\begin{equation}\label{commEq2}
\begin{split}
	\int_{\scrF_Q^\epsilon\cap[0,1]^{d-1}} & f\big(\vecx,n_-(\vecx)\Phi^{t},n_-(\vecx)\Phi^{t}\big) d\vecx \\ & = \int_{[0,1]^{d-1}} f\big(\vecx,n_-(\vecx)\Phi^{t},n_-(\vecx)\Phi^{t}\big) 
	\chi_\epsilon\big(n_+(-\vecx)\Phi^{-t}\big) d\vecx\\
	& = \int_{[0,1]^{d-1}} f\big(\vecx,n_-(\vecx)\Phi^{t},n_-(\vecx)\Phi^{t}\big) 
	\widetilde\chi_\epsilon\big(n_-(\vecx)\Phi^{t}\big) d\vecx .
\end{split}
\end{equation}

Case (A): If $\Gamma,\Gamma'$ are not commensurable, Theorem \ref{equiThm2} yields
\begin{multline}\label{last0}
	\lim_{t\to\infty} \int_{[0,1]^{d-1}}  f\big(\vecx,n_-(\vecx)\Phi^{t},n_-(\vecx)\Phi^{t}\big)\widetilde\chi_\epsilon\big(n_-(\vecx)\Phi^{t}\big)\, d\vecx  \\
= \int_{[0,1]^{d-1}\times \GamG\times \Gamma'\backslash G} \widetilde f(\vecx,M,M') \chi_\epsilon(M')\, d\vecx \, d\mu_{\Gamma}(M)\,d\mu_{\Gamma'}(M') .
\end{multline}

Case (B): If $\Gamma,\Gamma'$ are commensurable, then $\Delta=\Gamma\cap\Gamma'$ is a lattice in $G$ and Theorem \ref{equiThm0} yields
\begin{equation}\label{last}
\begin{split}
	\lim_{t\to\infty} & \int_{[0,1]^{d-1}}  f\big(\vecx,n_-(\vecx)\Phi^{t},n_-(\vecx)\Phi^{t}\big)\widetilde\chi_\epsilon\big(n_-(\vecx)\Phi^{t}\big)\, d\vecx  \\
	& =  \int_{[0,1]^{d-1}\times \Delta\backslash G} f(\vecx,M,M) \widetilde\chi_\epsilon(M)\, d\vecx \, d\mu_\Delta(M) \\
	& = \int_{[0,1]^{d-1}\times \Delta\backslash G} f(\vecx,\trans M^{-1},\trans M^{-1}) \chi_\epsilon(M)\, d\vecx \, d\mu_\Delta(M) \\
		& = \int_{[0,1]^{d-1}\times \Gamma'\backslash G} \overline f(\vecx,M,M) \chi_\epsilon(M)\, d\vecx \, d\mu_{\Gamma'}(M) ,
\end{split}
\end{equation}
since $d\mu_{\Gamma'}(M)=|\Gamma':\Delta|\, d\mu_{\Delta}(M)$.

{\bf Step 4: A volume computation.}
To evaluate the right hand sides of \eqref{last0} and \eqref{last}, we set (in order to treat both cases simultaneously)
\begin{equation}
	g(M)= \int_{[0,1]^{d-1}\times \GamG} \widetilde f(\vecx,\tilde M,M) \,  d\vecx \, d\mu_{\Gamma}(\tilde M) 
\end{equation}
for Case (A), and 
\begin{equation}
	g(M)= \int_{[0,1]^{d-1}} \overline f(\vecx,M,M) \, d\vecx 
\end{equation}
for Case (B). We thus need to evaluate
\begin{equation}\label{RhS}
\begin{split}
	 \int_{\Gamma'\backslash G} g(M) \chi_\epsilon(M)\, d\mu(M)
	 & = \int_{\Gamma_H\backslash G} g(M) \chi_\epsilon^1(M)\, d\mu(M)\\
	 & = \int_{\Gamma_H\backslash\scrH_\epsilon^1} g(M)\, d\mu(M) ,
\end{split}
\end{equation}
using \eqref{chi}.
Given $\vecy\in\RR^d$ we pick a matrix $M_\vecy\in G$ such that $(\vecnull,1) M_\vecy=\vecy$; recall \eqref{My} for an explicit choice of $M_\vecy$ for $y_d>0$. The map
\begin{equation}
	H \times \RR^d\setminus\{\vecnull\} \to  G, \qquad 
	(M, \vecy) \mapsto MM_\vecy, 
\end{equation}
provides a parametrization of $G$, where in view of \eqref{siegel}
\begin{equation}
	d\mu = \zeta(d)^{-1} d\mu_H \, d\vecy .
\end{equation}
Hence \eqref{RhS} equals, in view of \eqref{HaHa2},
\begin{equation}\label{RhS1}
	\frac{1}{\zeta(d)} \int_{\Gamma_H\backslash H \times\fC_\epsilon} g\big(M M_\vecy \big) \, d\mu_H(M) d\vecy .
\end{equation}

For
\begin{equation}\label{Ddef}
	D(y_d)=\begin{pmatrix} y_d^{-1/(d-1)} 1_{d-1} & \trans\vecnull \\ \vecnull & y_d \end{pmatrix},
\end{equation}
we have, in view of \eqref{dchoice},
\begin{equation}
	d\big(M_\vecy, D(y_d)\big)= d\big(D(y_d) n_+(y_d^{-1}\vecy'), D(y_d)\big)
	= d\big(n_+(y_d^{-1}\vecy'), 1_d \big)
	\leq y_d^{-1} \| \vecy' \| . 
\end{equation}
We recall that $y_d^{-1} \| \vecy' \|<\epsilon$ for $\vecy\in\fC_\epsilon$.
Therefore, with the choice of $\delta,\epsilon$ made in Steps 0 and 2, we have 
\begin{equation}\label{f0}
	\bigg| \eqref{RhS1} -
	\frac{1}{\zeta(d)} \int_{\Gamma_H\backslash H\times \fC_\epsilon} g\big(M D(y_d)\big) \, d\mu_H(M) d\vecy \bigg|  <   \frac{\delta}{\zeta(d)}\int_{\fC_\epsilon} d\vecy .
\end{equation}
Now,
\begin{equation}
\begin{split}
	\int_{\fC_\epsilon} g\big(M D(y_d) \big) \, d\vecy
	& = \vol(\scrB_1^{d-1})\,\epsilon^{d-1} \int_\theta^1 g\big(M D(y_d)\big)\, y_d^{d-1}\,dy_d \\
	& = (d-1) \vol(\scrB_1^{d-1})\,\epsilon^{d-1} \int_{0}^{|\log\theta|/(d-1)} g\big(M\Phi^{-s} \big) \,\e^{-d(d-1)s} ds ,
	\end{split} 
\end{equation}
and
\begin{equation}
	\int_{\fC_\epsilon} d\vecy = \frac{1}{d}\,\vol(\scrB_1^{d-1})\,\epsilon^{d-1} (1-\theta^d),
\end{equation}
where $\scrB_1^{d-1}$ denotes the unit ball in $\RR^{d-1}$.
So \eqref{f0} becomes
\begin{multline}\label{f1}
	\bigg| \eqref{RhS1} -
	\frac{(d-1) \vol(\scrB_1^{d-1})\,\epsilon^{d-1}}{\zeta(d)} \int_{0}^{|\log\theta|/(d-1)} \int_{\Gamma_H\backslash H}  g\big(M\Phi^{-s} \big) \, d\mu_H(M)\, \e^{-d(d-1)s} ds \bigg| \\ <  \frac{\vol(\scrB_1^{d-1})\,\delta \, \epsilon^{d-1}}{d\,\zeta(d)} \,(1-\theta^d)  .
\end{multline}

{\bf Step 5: Distance estimates.} Since \eqref{HaHaHa} is a disjoint union, we have furthermore (this is in effect another way of writing \eqref{commEq2} using \eqref{chi})
\begin{multline}
\int_{\scrF_Q^\epsilon\cap[0,1]^{d-1}} f\big(\vecx,n_-(\vecx)\Phi^{t},n_-(\vecx)\Phi^{t}\big) d\vecx \\
=	\sum_{\vecr\in\scrF_{Q,\theta}+\ZZ^{d-1}} \int_{\{\vecx\in [0,1]^{d-1}: \|\vecx-\vecr\|<\epsilon\e^{-dt}\}} f\big(\vecx,n_-(\vecx)\Phi^{t},n_-(\vecx)\Phi^{t}\big) d\vecx .
\end{multline}

Note that $n_-(\vecx)\Phi^t=\Phi^t n_-(\e^{dt}\vecx)$. By \eqref{dchoice}, for any $g\in G$,
\begin{equation}\label{expon}
	d\big(g n_-(\vecx)\Phi^t,g \Phi^t\big)
	= d\big(g \Phi^t n_-(\e^{dt} \vecx),g \Phi^t\big)
	= d\big(n_-(\e^{dt} \vecx), 1_d \big)
	\leq \e^{dt} \| \vecx \|. 
\end{equation}
Eq.~\eqref{expon} implies that 
\begin{equation}
	d\big(n_-(\vecx)\Phi^t,n_-(\vecr) \Phi^t\big)
	\leq \e^{dt} \| \vecx -\vecr \| <\epsilon .
\end{equation}
Because $f$ is uniformly continuous we therefore have, for the same $\delta,\epsilon$ as above:
\begin{multline}\label{f2}
\bigg| \int_{\|\vecx-\vecr\|<\epsilon\e^{-dt}} f\big(\vecx,n_-(\vecx)\Phi^{t},n_-(\vecx)\Phi^{t}\big) d\vecx - \frac{\vol(\scrB_1^{d-1}) \epsilon^{d-1}}{\e^{d(d-1)t}} f\big(\vecr,n_-(\vecr)\Phi^{t},n_-(\vecr)\Phi^{t}\big)  \bigg| \\ < \frac{\vol(\scrB_1^{d-1}) \,\delta\, \epsilon^{d-1}}{\e^{d(d-1)t}},
\end{multline}
uniformly for all $t\geq 0$.

{\bf Step 6: Conclusion.}
The approximations \eqref{f1} and \eqref{f2} hold uniformly for any $\delta>0$. Passing to the limit $\delta\to 0$, we obtain
\begin{multline}\label{lalalast}
	\lim_{t\to\infty} \frac{1}{\e^{d(d-1)t}} \sum_{\vecr\in\scrF_{Q,\theta}} f\big(\vecr,n_-(\vecr)\Phi^{t},n_-(\vecr)\Phi^{t}\big) \\
	= \frac{d-1}{\zeta(d)} \int_{0}^{|\log\theta|/(d-1)} \int_{\Gamma_H\backslash H}  g\big(M\Phi^{-s} \big) \, d\mu_H(M)\, \e^{-d(d-1)s} ds .
\end{multline}
The asymptotics \eqref{asymQ} 
shows that (recall that $Q^d=\e^{d(d-1)t}$)
\begin{equation}
	\limsup_{t\to\infty} \frac{|\scrF_Q\setminus \scrF_{Q,\theta}|}{\e^{d(d-1)t}}  \leq \frac{\theta^d}{d\,\zeta(d)} ,
\end{equation}
which allows us to take the limit $\theta\to 0$ in \eqref{lalalast}.
This concludes the proof for $\sigma=0$ and $f$ compactly supported, in both Case (A) and (B). For the extension to general $\sigma$ and $f$, recall the remarks in Step 0.
\end{proof}

\section*{Acknowledgements}

I thank Alexander Gorodnik and Andreas Str\"ombergsson for helpful discussions. This research project has been supported by a Royal Society Wolfson Research Merit Award, and a grant from the Max Planck Institute for Mathematics in Bonn, where this paper was written.

\end{document}